\documentclass[12pt]{amsart}
\usepackage[utf8]{inputenc}
\usepackage{bm}
\usepackage[dvipsnames,svgnames,x11names]{xcolor}
\usepackage[hyphens]{url}
\usepackage[colorlinks=true,
            linkcolor=Maroon,
            citecolor=blue,
            urlcolor=blue,
            hypertexnames=false,
            linktocpage]{hyperref}
\usepackage{bookmark}
\usepackage{amsmath,thmtools,mathtools,amssymb}
\mathtoolsset{showonlyrefs=true}
\usepackage{graphicx}
\usepackage{tikz}
\usepackage{fancyhdr}
\usepackage{esint}
\usepackage{enumerate}
\usepackage{caption}
\allowdisplaybreaks

\declaretheorem[name=Theorem,numberwithin=section]{thm}

\declaretheorem[name=Lemma,sibling=thm]{lemma}

\numberwithin{equation}{section}

\newcommand{\ti}{\tilde}

\newcommand{\wh}{\widehat}

\newcommand{\ov}{\overline}

\newcommand{\bbR}{\mathbb{R}}

\newcommand{\bbS}{\mathbb{S}}

\newcommand{\Sn}{\mathbb{S}^n}
\newcommand{\op}[1]{\operatorname{#1}}

\newcommand{\al}{\alpha}

\newcommand{\de}{\delta}

\newcommand{\Si}{\Sigma}

\newcommand{\cC}{\mathcal{C}}

\newcommand{\cR}{\mathcal{R}}

\newcommand{\del}{\partial}

\newcommand{\ip}[2]{\langle #1,#2 \rangle}

\newcommand{\eq}[1]{\begin{equation}\begin{alignedat}{2}#1\end{alignedat}\end{equation}}

\newcommand{\Cth}{\cC_\theta}

\begin{document}
\title{Capillary $L_p$-curvature problem}
\author[Y. Hu, M. N. Ivaki]{Yingxiang Hu,  Mohammad N. Ivaki}
\begin{abstract}
We prove a gradient estimate for a class of capillary curvature
equations in the half-space. As an application, we prove the existence of an even, smooth,
strictly convex solution to the even capillary $L_p$-curvature problem for all
$1<p<k+1$ and all contact angles $\theta\in(0,\pi/2)$.
\end{abstract}
\maketitle

\section{Introduction}
\begin{thm}\label{thm:Lp-curvature-problem}
Let $\theta \in (0,\frac{\pi}{2})$, $1 < p < k+1$ and $1 \leq k < n$. Suppose $0 < \phi \in C^{\infty}(\cC_\theta)$ is an even function. Then there exists an even, smooth, strictly convex capillary hypersurface $\Sigma \subset \overline{\bbR^{n+1}_{+}}$ whose $k$-th elementary symmetric function of the principal curvatures $S_k$ is given by
\eq{\label{capillary-CM-problem}
\phi  s^{p-1}S_k\circ\tilde{\nu}^{-1}= 1.
}
\end{thm}
In the closed (without boundary) setting, the case $p=1$ was solved by B. Guan and P. Guan \cite{GG02}. For $p> k+1$, existence was established in \cite{GRW15} by an elliptic method and in \cite{BIS21} via a parabolic approach, while the borderline case $p=k+1$ was treated in \cite{Lee23}. In \cite{HI24}, we resolved the remaining range $1<p<k+1$ by proving a new gradient estimate and combining it with a geometric lemma of Chou--Wang \cite{CW00} to obtain $C^0$ and $C^2$ estimates simultaneously.

\autoref{thm:Lp-curvature-problem} for $p=1$ and for $p\geq  k+1$ was obtained recently in \cite{MWW25c,MWW25d}. Furthermore, in their treatment of the range $1<p<k+1$, see \cite[Thm. 1.4]{MWW25d}, the authors proved existence under an \emph{additional} contact-angle restriction, depending on $p$.

The main new contribution of the present paper is a gradient estimate that removes this extra angle condition, thereby yielding the existence result for all $\theta\in(0,\pi/2)$. Our gradient estimate does not rely on any special structure of $S_k$ beyond its $k$-homogeneity; in particular, it applies to capillary quotient curvature problems, as well as to capillary $L_p$-Christoffel--Minkowski problem \cite{HI25b}, where for the latter problem we previously employed a completely different argument to avoid any $p$-dependent angle restriction (and the argument in \cite{HI25b} does not apply to the problem considered here). See also \cite{HIS25,MWW25dd}. For further results in the case $k=n$ and $p\geq -n-1$, we refer the reader to \cite{HI25,HHI25,MWW25a,MWW25b}.

\section{Background}
Let $\{e_i\}_{i=1}^{n+1}$ denote the standard orthonormal basis of $\bbR^{n+1}$. Set
\[
\bbR^{n+1}_+ := \{x\in\bbR^{n+1}: x_{n+1}>0\},
\quad
\del\bbR^{n+1}_+ := \{x\in\bbR^{n+1}: x_{n+1}=0\}.
\]
For $\theta\in(0,\pi)$, define the spherical cap and its translated version by
\[
\bbS^n_\theta := \{x\in\Sn:\ \ip{x}{e_{n+1}}\ge \cos\theta\},
\quad
\cC_\theta := \bbS^n_\theta-\cos\theta\,e_{n+1}.
\]

A smooth, compact, connected, orientable hypersurface
$\Si\subset\ov{\bbR^{n+1}_+}$ with $\op{int}(\Si)\subset\bbR^{n+1}_+$ and
$\del\Si\subset\del\bbR^{n+1}_+$ is called a \emph{capillary hypersurface with contact angle}
$\theta\in(0,\pi)$ if
$
\ip{\nu}{e_{n+1}}=\cos\theta
$
 on $\partial \Cth$ where $\nu$ denotes the outer unit normal of $\Si$.
 
We call $\Si$ \emph{strictly convex} if the enclosed region $\widehat\Si$ is a convex body (i.e. compact, convex, with non-empty interior) and the second fundamental form of $\Si$ is positive definite. For a strictly convex capillary hypersurface $\Si$, the \emph{capillary Gauss map} is defined by
\[
\ti\nu := \nu-\cos\theta\,e_{n+1}:\Si\to \cC_\theta.
\]
It is a diffeomorphism onto $\cC_\theta$, see \cite[Lem. 2.2]{MWWX25e}.

Let $\Si$ be strictly convex and $\theta$-capillary. The \emph{capillary support function} $s:\cC_\theta\to\bbR$ is defined by
\[
s(\zeta):=\ip{\ti\nu^{-1}(\zeta)}{\zeta+\cos\theta\,e_{n+1}}.
\]
Note that $s$ satisfies
$
\nabla_{\mu}s=\cot\theta \, s.
$

For the model cap $\cC_\theta$, the corresponding capillary support function is
\[
\ell(\zeta):=\sin^2\theta-\cos\theta\,\ip{\zeta}{e_{n+1}}.
\]

On $\cC_\theta$ we write $g$ for the round metric, $\nabla$ for its Levi--Civita connection, and
$\nabla^2$ for the covariant Hessian. For $f\in C^2(\cC_\theta)$ we set
\[
\tau[f]:=\nabla^2 f+f\,g,
\qquad
\tau^\sharp[f]:=g^{-1}\cdot\tau[f],
\]
so that $\tau^\sharp[f]$ is a symmetric endomorphism of $T\cC_\theta$.

Finally, writing $x\in\bbR^{n+1}$ as $x=(x_1,\dots,x_n,x_{n+1})$, let $\cR$ be the reflection
\[
\cR(x_1,\dots,x_n,x_{n+1}):=(-x_1,\dots,-x_n,x_{n+1}).
\]
A function $\varphi:\cC_\theta\to\bbR$ is called \emph{even} if $\varphi\circ\cR=\varphi$, and we
say that a capillary hypersurface $\Si$ (equivalently, its capillary support function $s$) is
\emph{even} if
\[
x\in\Si  \implies \cR(x)\in\Si.
\]

\section{The gradient estimate}
A key ingredient in the proof of \autoref{thm:Lp-curvature-problem} is the following gradient estimate in \autoref{lem: grad estimate}. Before stating the lemma and its proof, let us note that the auxiliary function
\[
\Phi=\frac{|\nabla u|^2}{u^{\gamma}}, \qquad u=\frac{s}{\ell},
\]
which is the natural capillary analogue of the test function used in \cite{HI24}, was employed in \cite{MWW25d} to obtain a gradient estimate. In the capillary setting, however, this choice turns out to be insufficient: it leads to a contact-angle restriction depending on $p$ when $1<p<k+1$. There are several natural candidates one might try in order to obtain a full gradient estimate. The main observation of the present paper is that the angle obstruction disappears once one incorporates the weight $\ell^{2-\gamma }$, namely by considering
\eq{\label{eq: defn-auxillary}
\Psi= \frac{\ell^{2-\gamma}|\nabla u|^2}{u^{\gamma}}.
}
In particular, when $\theta=\pi/2$, we have $\ell\equiv 1$ and hence $\Psi=|\nabla s|^2/s^{\gamma}$, which is precisely the auxiliary function used in \cite{HI24}.

\begin{lemma}[Main Lemma]\label{lem: grad estimate}
Let $1<p<k+1$, $\theta\in (0,\frac{\pi}{2})$ and $0<\gamma<\frac{2(p-1)}{k}$. Suppose $s$ is an even, smooth, strictly convex solution of
\eq{\label{s3:F-eq}
\left\{
\begin{aligned}
  F(\tau^\sharp[s])&=s^{p-1}\phi
    &&\text{in }\Cth,\\
  \nabla_\mu s &= \cot\theta\,s
    &&\text{on }\partial\Cth,
\end{aligned}
\right.
}
where $F$ is a symmetric, $k$-homogeneous, smooth curvature function. Then there exists $C_0=C_0(\theta,k,p,\gamma,\phi)$ such that
\eq{\label{s3:gradient-estimate}
\frac{|\nabla s|^2}{s^{\gamma}}\leq C_0 \left(\max_{\cC_\theta} s\right)^{2-\gamma}.
}
\end{lemma}
\begin{proof}
Let $u=s/\ell$, $0<\gamma<2(p-1)/k<2$ and $\beta=2-\gamma$. Consider the auxiliary function  $\Psi$ defined in \eqref{eq: defn-auxillary}.  To prove \eqref{s3:gradient-estimate}, it suffices to prove
\eq{\label{s3:grad-estimate-equivalent}
\frac{\ell^{\beta}|\nabla u|^2}{u^{\gamma}}
\leq
M \left(\max_{\cC_\theta}\ell\right)^\beta\left(\max_{\cC_\theta} u\right)^{2-\gamma}
}
for some constant $M>0$.
Assume that $\Psi$ attains its maximum at $\xi_0\in\cC_\theta$ and that
\eqref{s3:grad-estimate-equivalent} does not hold for a given $M>0$; that is,
\eq{
\frac{\ell^{\beta}|\nabla u|^2}{u^\gamma}(\xi_0)
=\max_{\xi\in \cC_\theta} \frac{\ell^{\beta}|\nabla u|^2}{u^\gamma}(\xi)> M \left(\max_{\cC_\theta}\ell\right)^\beta\left(\max_{\cC_\theta} u\right)^{2-\gamma} .
}
Then
\eq{
\frac{|\nabla u|^2(\xi_0)}{u^2(\xi_0)}> M.
}

In what follows, $c$ denotes a positive constant that may change from line to line, and depends only on $\theta,k,p,\gamma$ and $\phi$.

{\bf Case 1.} Suppose $\xi_0\in \del \cC_\theta$. Let $\{e_i\}_{i=1,\ldots,n}$ be an orthonormal frame around $\xi_0$ such that $e_n=\mu$. We use the shorthand notation $u_i:=\nabla_{e_i}u$ and $u_{ij}:=\nabla^2u(e_i,e_j)$, and similarly for higher covariant derivatives.
We have $\nabla_\mu u=u_n=0$ and
\eq{
\nabla^2 u(e_\al,e_n)=e_{\al}(u_n)-d u(\nabla_{e_\al}e_n)=-\cot\theta\, u_\al.
}
Since $\beta=2-\gamma<2$ and $\theta\in (0,\pi/2)$, at $\xi_0$ we obtain
\eq{
0\leq \nabla_\mu \log\left(\frac{\ell^\beta|\nabla u|^2}{u^{\gamma}}\right)
&=\frac{\nabla_\mu |\nabla u|^2}{|\nabla u|^2}+\beta\frac{\nabla_\mu\ell}{\ell}-\gamma \frac{\nabla_\mu u}{u}= -\gamma \cot\theta<0,
}
which is a contradiction.

{\bf Case 2.} Suppose $\xi_0\in \cC_\theta\backslash \del\cC_\theta$. At $\xi_0$, we have
\eq{\label{grad zero}
0=\nabla_i \log \Psi
&=\frac{\nabla_i(|\nabla u|^2)}{|\nabla u|^2}+\beta\frac{\ell_i}{\ell}-\gamma \frac{u_i}{u}\\
&=\frac{2u_m u_{mi}}{|\nabla u|^2}+\beta\frac{\ell_i}{\ell}-\gamma \frac{u_i}{u},
}
and
\eq{ \label{s3:F^ij-Psi_ij}
0\geq \nabla^2_{ij}(\log \Psi)
&=\frac{\nabla^2_{ij}(|\nabla u|^2)}{|\nabla u|^2}-\frac{\nabla_i (|\nabla u|^2)\nabla_j (|\nabla u|^2)}{|\nabla u|^4}\\
&\quad +\beta\frac{\ell_{ij}}{\ell}-\beta\frac{\ell_i\ell_j}{\ell^2}-\gamma \frac{u_{ij}}{u}+\gamma \frac{u_iu_j}{u^2}\\
&= \frac{2u_m u_{mij}+2u_{mi}u_{mj}}{|\nabla u|^2}-(\beta\frac{\ell_i}{\ell}-\gamma\frac{u_i}{u})(\beta\frac{\ell_j}{\ell}-\gamma\frac{u_j}{u})\\
&\quad +\beta\frac{1-\ell}{\ell}\de_{ij}-\beta\frac{\ell_i\ell_j}{\ell^2}-\gamma\frac{ u_{ij}}{u}+\gamma\frac{u_iu_j}{u^2},
}
where we used $\ell_{ij}+\ell \de_{ij}=\de_{ij}$.
Multiplying \eqref{s3:F^ij-Psi_ij} with $F^{ij}$ yields
\eq{ \label{s3:F^ij-u_m-u_{mij}}
0 \geq &~\frac{2}{|\nabla u|^2}(F^{ij}u_m u_{mij}+F^{ij}u_{mi}u_{mj})+\beta\frac{1-\ell}{\ell}F^{ij}\de_{ij}\\
&~+2\beta\gamma F^{ii}\frac{u_i \ell_i}{u\ell}-\gamma F^{ij}\left(\frac{\tau_{ij}}{u\ell}-\frac{2u_i\ell_j}{u \ell}-\frac{1}{\ell}\de_{ij}\right)\\
&~-(\beta+\beta^2)F^{ij}\frac{\ell_i\ell_j}{\ell^2}+(\gamma-\gamma^2)F^{ij}\frac{u_iu_j}{u^2} \\
= &~\frac{2}{|\nabla u|^2}(F^{ij}u_m u_{mij}+F^{ij}u_{mi}u_{mj})+\left(\beta\frac{1-\ell}{\ell}+\gamma\frac{1}{\ell}\right)F^{ij}\de_{ij}\\
&~-\gamma k s^{p-2}\phi+(2\beta\gamma+2\gamma) F^{ii}\frac{u_i \ell_i}{u\ell}-(\beta+\beta^2)F^{ii}\frac{\ell_i^2}{\ell^2}+(\gamma-\gamma^2)F^{ii}\frac{u_i^2}{u^2},
}
where we used
\eq{\label{tau-u}
\tau_{ij}=s_{ij}+s\de_{ij}
&=\ell u_{ij}+u_i \ell_j+u_j \ell_i+u(\ell_{ij}+\ell\de_{ij})\\
&=\ell u_{ij}+u_i \ell_j+u_j \ell_i+u\de_{ij}
}
and the identity
\eq{
F^{ij}(\ell u_{ij}+u_i \ell_j+u_j \ell_i+u\de_{ij})=F^{ij}\tau_{ij}=k F(\tau)= k s^{p-1} \phi.
}

Next, taking the first derivatives of \eqref{s3:F-eq} in the $e_m$-direction gives
\eq{
F^{ij}(\ell_m u_{ij}+\ell u_{ijm}+2u_{im}\ell_j+2u_i\ell_{jm}+u_m\de_{ij})=(p-1)s^{p-2}\phi s_m +s^{p-1}\phi_m.
}
Multiplying this identity by $u_m/\ell$, summing over $m$, and using \eqref{tau-u} and \eqref{grad zero}, we obtain
\eq{
u_mF^{ij} u_{ijm}
=&~\frac{1}{\ell}((p-1)s^{p-2}\phi s_m u_m +s^{p-1}u_m \phi_m)\\
&~-\frac{1}{\ell}F^{ij}(u_m \ell_m u_{ij}+2u_{im}u_m\ell_j+2u_i\ell_{jm}u_m+|\nabla u|^2\de_{ij})\\
=&~(p-1)s^{p-2}\left(|\nabla u|^2+\frac{u}{\ell}\ell_mu_m\right) \phi+s^{p-1}\frac{u_m}{\ell} \phi_m\\
&~-\frac{\ell_m u_m }{\ell^2}(ks^{p-1}\phi-2F^{ii}u_i\ell_i-uF^{ij}\de_{ij})\\
&~-|\nabla u|^2F^{ii}\frac{\ell_i}{\ell} \left(\gamma\frac{u_i}{u}-\beta\frac{\ell_i}{\ell}\right)-\frac{2(1-\ell)}{\ell}F^{ii}u_i^2-\frac{|\nabla u|^2}{\ell} F^{ij}\de_{ij}\\
=&~(p-1)s^{p-2}|\nabla u|^2 \phi+(p-1-k)s^{p-2}\frac{u}{\ell}\ell_m u_m \phi+s^{p-1}\frac{u_m}{\ell} \phi_m\\
&~+\left(\frac{u\ell_m u_m }{\ell^2} -\frac{|\nabla u|^2}{\ell}\right)F^{ij}\de_{ij}+\left(\frac{2\ell_mu_m}{\ell^2}-\gamma\frac{|\nabla u|^2}{u\ell}\right)F^{ii}u_i\ell_i\\
&~+\beta|\nabla u|^2F^{ii}\frac{\ell_i^2}{\ell^2}-\frac{2(1-\ell)}{\ell}F^{ii}u_i^2.
}

Since $\nabla^2u+gu$ is Codazzi, we have
$
u_{mij} =u_{ijm}+u_m \de_{ij}-u_j\de_{mi},
$
and thus
\eq{
F^{ij}u_mu_{mij}
= &~ F^{ij} u_m(u_{ijm}+u_m \de_{ij}-u_j\de_{mi}) \\
=&~F^{ij}u_mu_{ijm}+|\nabla u|^2 F^{ij}\de_{ij}-F^{ii}u_i^2\\
=&~(p-1)s^{p-2}|\nabla u|^2\phi+(p-1-k)s^{p-2}\frac{u}{\ell}\ell_m u_m \phi+s^{p-1}\frac{u_m}{\ell} \phi_m\\
&~+\left(\frac{u\ell_m u_m }{\ell^2} +(1-\frac{1}{\ell})|\nabla u|^2\right)F^{ij}\de_{ij}+\left(\frac{2\ell_mu_m}{\ell^2}-\gamma\frac{|\nabla u|^2}{u\ell}\right)F^{ii}u_i\ell_i\\
&~+\beta|\nabla u|^2F^{ii}\frac{\ell_i^2}{\ell^2}-\frac{2-\ell}{\ell}F^{ii}u_i^2.
}
Therefore, by \eqref{s3:F^ij-u_m-u_{mij}}, we find
\eq{
0 \geq &~\frac{2}{|\nabla u|^2}(F^{ij}u_m u_{mij}+ F^{ii}u_{mi}^2)+\left(\beta\frac{1-\ell}{\ell}+\gamma\frac{1}{\ell}\right)F^{ij}\de_{ij}-\gamma k s^{p-2}\phi\\
&~+(2\beta\gamma+2\gamma) F^{ii}\frac{u_i \ell_i}{u\ell}-(\beta+\beta^2)F^{ii}\frac{\ell_i^2}{\ell^2}+(\gamma-\gamma^2)F^{ii}\frac{u_i^2}{u^2}\\
=&~(2(p-1)-\gamma k)s^{p-2}\phi+2\frac{u_m}{\ell|\nabla u|^2} s^{p-1}\phi_m+2(p-1-k)\frac{\ell_m u_m }{\ell^2|\nabla u|^2}s^{p-1}\phi\\
&~+\left(\frac{2u\ell_m u_m }{\ell^2|\nabla u|^2} +(2-\beta)(1-\frac{1}{\ell})+\gamma\frac{1}{\ell}\right)F^{ij}\de_{ij}+\left(\frac{4u\ell_mu_m}{\ell|\nabla u|^2}+2\beta\gamma\right)F^{ii}\frac{u_i\ell_i}{u\ell}\\
&~+(\beta-\beta^2) F^{ii}\frac{\ell_i^2}{\ell^2}+\frac{2}{|\nabla u|^2}F^{ii}u_{mi}^2+\left(\gamma-\gamma^2-\frac{2(2-\ell)}{\ell}\frac{u^2}{|\nabla u|^2}\right)F^{ii}\frac{u_i^2}{u^2}.
}
Moreover, by \eqref{grad zero} we have
\eq{
\frac{2F^{ii} u_{mi}^2}{|\nabla u|^2}
\geq \frac{F^{ii}}{2}\left(\frac{2u_m u_{mi}}{|\nabla u|^2}\right)^2
=\frac{\gamma^2}{2}F^{ii}\frac{u_i^2}{u^2}
+\frac{\beta^2}{2}F^{ii}\frac{\ell_i^2}{\ell^2}
-\gamma\beta F^{ii}\frac{u_i\ell_i}{u\ell}.
}
Hence,
\eq{
0 \geq &~s^{p-2}\phi\left( 2(p-1)-k\gamma-2(k+1-p)\frac{u|\nabla\ell|}{\ell |\nabla u|}-2\frac{u|\nabla \phi|}{\phi|\nabla u|}\right)\\
&~+\left(-\frac{2u|\nabla\ell|}{\ell^2|\nabla u|} +(2-\beta)(1-\frac{1}{\ell})+\gamma\frac{1}{\ell}\right)F^{ij}\de_{ij}
-\left(\frac{4u|\nabla \ell|}{\ell|\nabla u|}+\beta\gamma\right)F^{ii}\frac{|u_i||\ell_i|}{u\ell}\\
&~+\left(\beta-\frac{\beta^2}{2}\right) F^{ii}\frac{\ell_i^2}{\ell^2}
+\left(\gamma-\frac{\gamma^2}{2}-\frac{2(2-\ell)}{\ell}\frac{u^2}{|\nabla u|^2}\right)F^{ii}\frac{u_i^2}{u^2}\\
\geq &~s^{p-2}\phi\left( 2(p-1)-k\gamma-\frac{c}{\sqrt{M}}\right)
+\left(-\frac{c}{\sqrt{M}} +(2-\beta)(1-\frac{1}{\ell})+\gamma\frac{1}{\ell}\right)F^{ij}\de_{ij}\\
&~-\left(\frac{c}{\sqrt{M}}+\beta\gamma\right) F^{ii}\frac{|u_i||\ell_i|}{u\ell}
+\left(\beta-\frac{\beta^2}{2}\right) F^{ii}\frac{\ell_i^2}{\ell^2}
+\left(\gamma-\frac{\gamma^2}{2}-\frac{c}{M}\right)F^{ii}\frac{u_i^2}{u^2}.
}

Let $\beta=2-\gamma$. Then
\eq{
0\geq &~s^{p-2}\phi\left( 2(p-1)-k\gamma-\frac{c}{\sqrt{M}} \right)
+\left( -\frac{c}{\sqrt{M}}+\gamma\right)F^{ij}\de_{ij}\\
&~-\left(\frac{c}{\sqrt{M}}+2(\gamma-\frac{\gamma^2}{2})\right)F^{ii}\frac{|u_i||\ell_i|}{u\ell}\\
&~+\left(\gamma-\frac{\gamma^2}{2}\right)F^{ii}\frac{\ell_i^2}{\ell^2}
+\left(\gamma-\frac{\gamma^2}{2}-\frac{c}{M}\right)F^{ii}\frac{u_i^2}{u^2}.
}

Let $\varepsilon\in (0,1)$. Choose $M>0$ sufficiently large such that
\begin{equation}
\begin{alignedat}{2}
\frac{c}{\sqrt{M}}&<\frac{1}{2}\left(2(p-1)-k\gamma\right),\quad
&\frac{c}{\sqrt{M}}&<\frac{\gamma}{2},\\
\frac{c}{\sqrt{M}}&<2(\sqrt{1+\varepsilon}-1)\left(\gamma-\frac{\gamma^2}{2}\right),\quad
&\frac{c}{M}&<\varepsilon\left(\gamma-\frac{\gamma^2}{2}\right).
\end{alignedat}
\end{equation}
Using $\ell_i^2\leq |\nabla\ell|^2 \leq c\,\ell^2$ and the Cauchy--Schwarz inequality
\eq{
2\sqrt{1+\varepsilon} F^{ii}\frac{|u_i||\ell_i|}{u\ell}
\leq \frac{1+\varepsilon}{1-\varepsilon} F^{ii}\frac{\ell_i^2}{\ell^2}
+(1-\varepsilon)F^{ii}\frac{u_i^2}{u^2},
}
we obtain
\eq{
0\geq &~\frac{2(p-1)-k\gamma}{2}s^{p-2}\phi+\frac{\gamma}{2}F^{ij}\de_{ij}
-2\sqrt{1+\varepsilon}\left(\gamma-\frac{\gamma^2}{2}\right)F^{ii}\frac{|u_i||\ell_i|}{u\ell}\\
&~+\left(\gamma-\frac{\gamma^2}{2}\right)F^{ii}\frac{\ell_i^2}{\ell^2}
+(1-\varepsilon)\left(\gamma-\frac{\gamma^2}{2}\right)F^{ii}\frac{u_i^2}{u^2}\\
\geq &~\frac{2(p-1)-k\gamma}{2}s^{p-2}\phi+\frac{\gamma}{2}F^{ij}\de_{ij}
-\left(\gamma-\frac{\gamma^2}{2}\right)\frac{2\varepsilon}{1-\varepsilon}F^{ii}\frac{\ell_i^2}{\ell^2}\\
\geq &~\frac{2(p-1)-k\gamma}{2}s^{p-2}\phi+\frac{\gamma}{2}\left(1-\frac{4c\varepsilon}{1-\varepsilon}\right)F^{ij}\de_{ij}.
}
Therefore, if $\varepsilon>0$ is sufficiently small such that $1-\frac{4c\varepsilon}{1-\varepsilon}>0$, then we arrive at a contradiction.
\end{proof}

Regarding the proof of \autoref{thm:Lp-curvature-problem}, the gradient estimate, combined with the arguments of \cite{MWW25d}, yields two-sided bounds for the capillary support function and a uniform $C^2$ estimate simultaneously. Higher-order regularity then follows from \cite{Lie13}. One can therefore apply the degree-theoretic framework of \cite{LLN17}, together with the uniqueness result \cite[Thm. 1.6]{GL24}, to obtain the desired existence result; see \cite{MWW25d} for further details.

Next, we present a new argument, which is of independent interest, showing that the gradient estimate directly implies two-sided bounds for solutions to the even capillary $L_p$-Christoffel-Minkowski problem. See also \cite{HI25b} for a more convex-geometric argument.

\begin{lemma}\label{lem:grad-base}
Let $\Sigma$ be an even, smooth, strictly convex $\theta$-capillary hypersurface with capillary support function $s$. Let
\[
\Omega:=\widehat{\Sigma}\cap e_{n+1}^\perp\cong\bbR^n
\]
and denote by $h:\bbS^{n-1}\to\bbR$ the (standard) support function of $\Omega$
(viewed as a convex body in $e_{n+1}^\perp$).
Assume that for some $\gamma\in(0,2)$ we have
\eq{\label{eq:cap-grad}
\frac{|\nabla s|^2}{s^\gamma}\le C_0,R^{2-\gamma},
\quad
R:=\max_{\cC_\theta}s.
}
Let $\wh{\nabla}=\nabla_{\bbS^{n-1}}$. Then $h$ satisfies
\eq{\label{eq:base-grad}
\frac{|\wh{\nabla}h|^2}{h^\gamma}
\le C_0\,(\sin\theta)^{\gamma}\,R^{2-\gamma} \quad  \text{on } \bbS^{n-1}.
}
\end{lemma}

\begin{proof}
Let $\hat{s}:\bbS^n\to\bbR$ be the (standard) support function of $\widehat{\Sigma}$.
For $u\in\bbS^n_\theta$:
\eq{
s(\zeta)=\hat{s}(u)
\quad\text{whenever}\quad
u=\zeta+\cos\theta\,e_{n+1}
\, (\text{i.e.\ }\zeta=u-\cos\theta\,e_{n+1}\in\cC_\theta).
}
Moreover, for $\tilde u\in\bbS^{n-1}\cong \bbS^n\cap e_{n+1}^\perp$ one has
\eq{\label{eq:omega-support}
h(\tilde u)
=
\frac{1}{\sin\theta}\,\hat{s}(\sin\theta\,\tilde u+\cos\theta\,e_{n+1}).
}
Therefore,
\eq{\label{eq:omega-support-via-s}
h(\tilde u)
=
\frac{1}{\sin\theta}\,s(\sin\theta\,\tilde u), \quad
|\wh{\nabla}h(\tilde u)|
\le
|\nabla s(\sin\theta\,\tilde u)|,
}
and hence
\[
|\wh{\nabla}h|^2
\le
C_0\,R^{2-\gamma}\,(\sin\theta)^\gamma\,h^\gamma.
\]
\end{proof}

\begin{thm}
Let $1<p<k+1$, $\theta\in (0,\frac{\pi}{2})$ and $0<\gamma<\frac{2(p-1)}{k}$. Suppose $s$ is an even, smooth, strictly convex solution of
\eq{\label{equation pde}
\left\{
\begin{aligned}
 \sigma_k(\tau^\sharp[s])&=s^{p-1}\phi
    &&\text{in }\Cth,\\
  \nabla_\mu s &= \cot\theta\,s
    &&\text{on }\partial\Cth,
\end{aligned}
\right.
}
where $\sigma_k$ is the $k$-th elementary symmetric function of principal radii of curvature. Then there exists $C_1=C_1(\theta,k,p,\gamma,\phi)$ such that
\eq{
C_1^{-1}\leq s\leq C_1.
}
\end{thm} 
\begin{proof}
Let $\Sigma$ be the even, smooth, strictly convex, $\theta$-capillary hypersurface with capillary support function $s$.
Set $\Omega=\wh{\Si}\cap e_{n+1}^{\perp}$ and
\eq{
r_{\mathrm{out}}:=\max_{x'\in\Omega}|x'|,
\quad
r_{\mathrm{in}}:=\min_{x'\in\partial\Omega}|x'|,
\quad
H:=\max_{x\in \Sigma} x_{n+1},\quad R:=\max_{\Cth}s.
}

By \cite[Thm. 2.9]{HI25b} we have
\eq{
R \le r_{\mathrm{out}}+H
\le r_{\mathrm{out}}+\tan\theta\, r_{\mathrm{in}}
\le (1+\tan\theta) r_{\mathrm{out}}
=(1+\tan\theta) \max_{\mathbb{S}^{n-1}} h,
}
where $h$ denotes the standard support function of $\Omega$. Thus, by \autoref{lem:grad-base},
\eq{
\frac{|\wh{\nabla}h|^2}{h^\gamma}
\le C' (\sin\theta)^{\gamma} \left(\max_{\mathbb{S}^{n-1}} h\right)^{2-\gamma}
\qquad\text{on }\mathbb{S}^{n-1}.
}
Due to \cite[Lem. 3.1]{Gua22}, there exists $c=c(C',\gamma,\theta)>0$,
such that
\eq{
r_{\mathrm{out}}\le c\, r_{\mathrm{in}}.
}

Moreover, directly from \eqref{equation pde} (cf. \cite[Lem. 3.3]{HI25b}), it follows that there exists $c_0>0$ (depending only on $\theta$ and $\phi$) such that
\eq{
R\ge c_0.
}
Combining this with $r_{\mathrm{out}}\le c\,r_{\mathrm{in}}$  gives
\eq{
c_0\le R\le r_{\mathrm{out}}+H\le (c+\tan\theta)r_{\mathrm{in}}\Rightarrow
r_{\mathrm{in}}\ge \frac{c_0}{c+\tan\theta}  \text{ and } R\leq  (c+\tan\theta)r_{\mathrm{in}}.
}

On the other hand, the argument in \cite[Sec. 3.1]{HI25b} yields
\eq{
H\ge \frac{c_k^{1/k}}{2}\phi_1^{-1/k} R^{\frac{1-p}{k}} r_{\mathrm{in}}^2,
}
where $\phi_0\le \phi\le \phi_1$ and hence
\eq{
H\ge c_1 r_{\mathrm{in}}^{2+\frac{1-p}{k}}\ge c_2,
}
for constants $c_1,c_2>0$ depending only on $k,p,\theta,\phi$ and $c, c_0$.
In particular, we obtain the uniform lower bound
\eq{
s\ge H\cos\theta \ge c_2\cos\theta .
}

Finally, since $1<p<k+1$, the inequalities
\eq{
c_1 r_{\mathrm{in}}^{2+\frac{1-p}{k}}\le H\le \tan\theta\,r_{\mathrm{in}}
}
force an upper bound $r_{\mathrm{in}}\le c_3$ for some $c_3=c_3(c_1,\theta)>0$.
Therefore,
\eq{
s\le (c+\tan\theta)r_{\mathrm{in}}\le (c+\tan\theta)c_3.
}
\end{proof}

\section*{Acknowledgments}
Hu was supported by the National Key Research and Development Program of China (Grant No. 2021YFA1001800). Ivaki was supported by the Austrian Science Fund (FWF) under Project P36545.

\vspace{5mm}
\textsc{School of Mathematical Sciences, Beihang University,\\ Beijing 100191, China,\\}
\email{\href{mailto:huyingxiang@buaa.edu.cn}{huyingxiang@buaa.edu.cn}}

\vspace{5mm}
\textsc{Institut f\"{u}r Diskrete Mathematik und Geometrie,\\ Technische Universit\"{a}t Wien,\\ Wiedner Hauptstra{\ss}e 8-10, 1040 Wien, Austria,\\}
\email{\href{mailto:mohammad.ivaki@tuwien.ac.at}{mohammad.ivaki@tuwien.ac.at}}

\end{document}